\documentclass[12pt,draft]{amsart}
\usepackage[all]{xy}
\usepackage{amsfonts}
\usepackage{amssymb}

\usepackage[english]{babel}
\usepackage{enumerate}

\usepackage{color}

\newtheorem{teo}{Theorem}[section]
\newtheorem{lem}[teo]{Lemma}

\newtheorem{cor}[teo]{Corollary}

\newtheorem{dfn}[teo]{Definition}
\newtheorem{rmk}[teo]{Remark}

\def\<{\langle}
\def\>{\rangle}

\def\a{\alpha}

\def\g{\gamma}

\def\t{\tau}

\def\f{{\varphi}}

\def\G{{\Gamma}}

\def\Z{{\mathbb Z}}

\def\Im{\mathop{\rm Im}\nolimits}

\def\Coker{\operatorname{Coker}}
\def\Id{\operatorname{Id}}

\def\1{\mathbf 1}

\newcommand{\ov}[1]{\overline{#1}}

\begin{document}

\title[Reidemeister in wreath products]
{Reidemeister classes in some wreath products by $\Z^k$}

\author{M.I. Fraiman} 
\author{E.V. Troitsky}
\thanks{The work was supported by the Foundation for the
	Advancement of Theoretical Physics and Mathematics ``BASIS''}
\address{Moscow Center for Fundamental and Applied Mathematics, MSU Department,\newline
	Dept. of Mech. and Math., Lomonosov Moscow State University, 119991 Moscow, Russia}
\email{troitsky@mech.math.msu.su}

\begin{abstract}
Among restricted wreath products $G\wr \Z^k $,  where $G$ is a finite Abelian group, we find three large classes of groups admitting an automorphism $\f$ with finite Reidemeister number $R(\f)$ (number of $\f$-twisted conjugacy classes). In other words, groups from these classes do not have the $R_\infty$ property. 

If a general automorphism $\f$ of $G\wr \Z^k$ has a finite order (this is the case for $\f$ detected in the first part of the paper) 
and $R(\f)<\infty$, 
we prove that $R(\f)$ coincides with the number of equivalence classes of finite-dimensional irreducible unitary representations
of  $G\wr \Z^k$, which are fixed by the dual map $[\rho]\mapsto [\rho\circ \f]$
(i.e. we prove the conjecture about finite twisted Burnside-Frobenius theorem, TBFT$_f$, for these $\f$).
\end{abstract}

\maketitle

\section{Introduction}\label{sec:intro}
Suppose, $\G$ is a  group and $\phi: \G\rightarrow \G$ is an endomorphism.
Two elements $x,y\in \G$ are 
 $\phi$-{\em conjugate} or {\em twisted conjugate,}
if and only if there exists an element $g \in \G$ such that
$$
y=g  x   \phi(g^{-1}).
$$
The corresponding classes are called \emph{Reidemeister} or 
{\em twisted conjugacy} classes.
The number $R(\phi)$ of them is called the {\em Reidemeister number}
of $\phi$.

The study of Reidemeister numbers is an important problem related with Topological Dynamics, Number Theory and Representation Theory (see \cite{FelshB}). One of the main problems in the field is to prove or disprove the so-called TBFT (a conjecture about the twisted Burnside-Frobenius theory (or theorem)), which has numerous important consequence for Reidemeister zeta function and 
for other problems in Topological Dynamics (see a more extended discussion in \cite{FeTrZi20}). Namely the problem is to identify $R(\f)$ (when $R(\f)<\infty$) in a natural way with the number of fixed points of the induced map $\widehat{\f}$ of an appropriate dual object.  In the initial formulation of the conjecture \cite{FelHill}, the dual object was the unitary dual $\widehat{\G}$ and
$\widehat{\f}:[\rho]\mapsto [\rho\circ \f]$. The conjecture about TBFT was proved in many cases, but failed for an example in \cite{FelTroVer}, which led to the new formulation: TBFT$_f$, where $\widehat{\G}$ was replaced by its finite-dimensional part,
which is evidently invariant under $\widehat{\f}$. This is the version, which we will study in this paper for a class of groups.
In \cite{FeTrZi20} an example of a group that has neither TBFT nor TBFT$_f$ was presented.
The most general proved cases of TBFT$_f$ are the case of polycyclic-by-finite groups \cite{polyc}
and the case of nilpotent torsion-free groups of finite Pr\"ufer rank \cite{FelTroDicjtomy2021RJMP}. 

Another important problem in the field is to localize the class of groups, where one can consider the TBFT conjecture, i.e. where
automorphisms with $R(\f)<\infty$ do exist. The opposite case is called the $R_\infty$ property.
It has some topological consequences itself (see e.g. \cite{GoWon09Crelle}).
A part of recent results about Reidemeister clases and $R_\infty$ can be found in
 \cite{FelLeonTro,Romankov,BardNasyNes2013,DekimpeGoncalves2014BLMS,FelTroJGT,TroLamp,TroitskyWBranch2019RJMP,Nasybullov2020TMNA}
(see also an overview in \cite{FelNasy2016JGT}). 

We consider the following restricted wreath product $G\wr \Z^k = \Sigma \rtimes_\a \Z^k$,  where $G$ is a finite Abelian group, 
$\Sigma$ denotes $\oplus_{x\in \Z^k} G_x$, and $\a(x)(g_y) =g_{x+y}$. Here $g_x$ is $g \in G \cong G_x$.

The $R_\infty$ property was completely studied for $k=1$ in \cite{gowon1}, for $G=\Z_p$ with a  prime $p$ and arbitrary $k$ 
in \cite{TroLamp}, for $G=\Z_m$  and arbitrary $k$ in \cite{Fraiman}. In all these cases the TBFT$_f$ was proved. 

The complexity of the study increases drastically when we move from $k=1$ to $k>1$, 
because $\Z$ has only one non-trivial automorphism in contrast with $\Z^k$. 

The groups under consideration can be viewed as generalized lamplighter groups.
For a generalization of the lamplighter group in other directions, the twisted conjugacy was considered in \cite{TabackWong2011},
\cite{SteinTabackWong2015}, and other papers.

In the present paper, we prove (Theorem \ref{teo:cases}) that the groups under consideration do not have the $R_\infty$ property in the following three cases:
\begin{enumerate}[1)]
\item all prime-power components of $G$ for $2$ and $3$ have multiplicity at least $2$;
\item there is no prime-power components for $2$  and $k$ is even;
\item all prime-power components of $G$ for $2$ have multiplicity at least $2$ and $k=4s$ for some $s$.
\end{enumerate}

To prove this, we construct corresponding examples, and all of them have a finite order. This motivates us to prove the TBFT$_f$ for all groups of the form $G\wr \Z^k$ and their automorphisms of finite order (Corollary \ref{cor:main_TBFT}).
The proof is based on a description of Reidemeister classes of $\f$ as cylindrical sets (Theorem \ref{teo:main_for_TBFT}). 

\textbf{Acknowledgment.} The work was supported by the Foundation for the Advancement of Theoretical
Physics and Mathematics “BASIS”.

\section{Preliminaries}

We start from some general statements about Reidemeister classes of extensions.
Suppose,  a normal subgroup $H$ of $G$ is $\f$-invariant under an automorphism $\f:G\to G$
and $p:G\to G/H$ is the natural projection.
Then $\f$ induces automorphisms $\f':H\to H$ and $\widetilde{\f}:G/H \to G/H$.

\begin{dfn}
	\rm Denote $C(\f):=\{g\in G \colon \f(g)=g\}$, i.e. $C(\f)$ is the subgroup of $G$, formed by $\f$-fixed elements.
\end{dfn}

We will use the notation $\tau_g(x)=gxg^{-1}$ for an inner automorphism as well as for its restriction on a normal subgroup.

The following important properties were obtained in \cite{FelHill,go:nil1}, see also \cite{polyc,GoWon09Crelle}.

\begin{teo}\label{teo:extensions}
	For $G$, $H$, $\f$, $\f'$, and $\widetilde{\f}$ as above, we have the following.
	\begin{itemize}
		\item[1.] \emph{Epimorphity:} the projection $G\to G/H$ maps Reidemeister classes of $\f$ onto Reidemeister classes of $\widetilde{\f}$, in particular $R(\widetilde{\f})\le R(\f)$;
		\item[2.] \emph{Estimation by fixed elements:} if $|C(\widetilde{\f})|=n$, then $R(\f')\le R(\f)\cdot n$;
		\item[3.] \emph{Fixed elements-free case:} if $C(\widetilde{\f})=\{e\}$, then 		each Reidemeister class of $\f'$ is an intersection of the appropriate Reidemeister class of $\f$ and $H$; 
		\item[4.] \emph{Summation:} if $C(\widetilde{\f})=\{e\}$, then 
		$R(\f)=\sum_{j=1}^R R(\tau_{g_j} \circ \f')$, where $g_1,\dots g_R$ are some elements of $G$ such that
		$p(g_1),\dots,p(g_R)$ are representatives of all Reidemeister classes of $\widetilde{\f}$, $R=R(\widetilde{\f})$;
	\end{itemize}	
\end{teo}

Also we will need the following statement \cite{Jabara} (Lemma 4 and the step (2) in the proof of Theorem A'):
\begin{lem}\label{lem:Jab_fin_ord}
Suppose, $\G$ is a residually finite group and $\f:\G\to\G$  is an automorphism with $R(\f)<\infty$. Then
$|C(\f)|<\infty$.
\end{lem}
One can find in \cite{Jabara} an estimation for $|C(\f)|$, but we will not use it.

Passing to a semidirect product $\Sigma \rtimes_\a \Z^k $, we have by \cite{Curran2008} that a couple of automorphisms $\f':\Sigma\to\Sigma$
and $\overline{\f}: \Sigma \rtimes_\a \Z^k  /\Sigma \cong \Z^k \to \Z^k \cong \Sigma \rtimes_\a \Z^k  /\Sigma$ define an automorphism
$\f$ of $\Sigma \rtimes_\a \Z^k$ (not unique) if and only if
\begin{equation}\label{eq:maineq}
	\f'(\a(m)(h))=\a(\ov{\f}(g))(\f'(h)),\qquad h\in\Sigma,\quad m\in \Z^k.
\end{equation}
Since $\Sigma$ is abelian, by \cite[p.~207]{Curran2008} the mapping $\f_1$ defined as $\f'$ on $\Sigma$ and by $\overline{\f}$ on $\Z^k\subset \Sigma \rtimes \Z^k $ is still an automorphism.  Moreover,  from the following commutative diagrams
$$
\xymatrix{0 \ar[r]& \Sigma \ar[r] \ar[d]_{\f'}& \Sigma \rtimes \Z^k \ar[r] \ar@/_/[d]_{\f} \ar@/^/[d]^{\f_1} & \Z^k \ar[r] \ar[d]^{\overline{\f}}& 0\\
0 \ar[r]& \Sigma \ar[r] & \Sigma \rtimes \Z^k \ar[r] & \Z^k \ar[r] & 0}
$$
we have $R(\f)=R(\f_1)$. Indeed, if $R(\overline{\f})=\infty$ then $R(\f)=R(\f_1)=\infty$. If $R(\overline{\f})<\infty$ then $C_{\overline{\f}}=\{0\}$ 
and by Theorem \ref{teo:extensions}
$$
R(\f)=\sum_{\mbox{\scriptsize representatives }m\in \Z^k\mbox{\scriptsize of Reidemeister classes of }\overline{\f}}  R(\t_m \circ \f') =  R(\f_1).
$$ 
So, without loss of generality in the $R_\infty$ questions (not in Section \ref{sec:tbft}) we will assume
\begin{equation}\label{eq:restric_on_sub}
\Z^k \subset A\wr \Z^k \mbox{ is $\f$-invariant and } \f|_{\Z^k}=\overline{\f}.
\end{equation}
This was discussed briefly in \cite[Lemma 3.5]{gowon1} in a particular case.

\begin{lem}\label{lem:R_needed}
An authomorphism $\f: G\wr \Z^k \to G \wr \Z^k$ has $R(\f)<\infty$ if and only if 
$R(\overline{\f})< \infty$ and $R(\t_m \circ \f')<\infty$ for any $m \in \Z^k$ (in fact, it is sufficient to verify this for representatives
of Reidemeister classes of $\overline{\f}$).
\end{lem}

\begin{proof}
Suppose, $R(\f)<\infty$.
By Theorem \ref{teo:extensions}, we have $R(\overline{\f})<\infty$. Then by Lemma  \ref{lem:Jab_fin_ord}, we obtain
$|C(\overline{\f})|<\infty$ (in fact, $|C(\overline{\f})|=1$, because an automorphism of $\Z^k$ can not have finitely many fixed elements except of $0$). So, by Theorem \ref{teo:extensions}, $R(\f')<\infty$. Considering $\t_z \circ \f$, which has $R(\t_z \circ \f)=R(\f)<\infty$, instead of $\f$, we obtain in the same way that $R(\t_z \circ \f')<\infty$.

 Conversely, having  $|C(\overline{\f})|=1$, one can apply the summation formula from Theorem \ref{teo:extensions}.
\end{proof}

\begin{lem}\label{lem:how_to_define}
Suppose, $\overline{\f}:\Z^k \to \Z^k$ is an automorphism and $F:G\to G$ is an automorphism.
Then $\f'$ defined by
\begin{equation}\label{eq:how_to_def}
\f'(a_0)=(Fa)_0,\qquad \f'(a_x)=(Fa)_{\overline{\f}(x)}
\end{equation}
satisfies (\ref{eq:maineq}) and so defines an automorphism of $G\wr \Z^k$.

Evidently the subgroups $\oplus G_x$, where $x$ runs over an orbit of $\overline{\f}$, are $\f'$-invariant summands of $\Sigma$.
\end{lem}

\begin{proof}
It is sufficient to prove (\ref{eq:maineq}) on generating elements of the form $a_x$. Then for any $z\in \Z^k$,
$$
\f'(\a(z) a_x)=\f'(a_{x+z})= (Fa)_{\overline{\f}(x+z)}=\a(\overline{\f}(z)) (Fa)_{\overline{\f}(z)}=\a(\overline{\f'}(z)) \f'(a_x)
$$
and  (\ref{eq:maineq})  is fulfilled. The first equality in (\ref{eq:how_to_def}) is in fact a particular case of the second one.
\end{proof}

It is not difficult to prove (see \cite{FelshtynHill1993CM}) that,
for $\overline{\f}:\Z^k\to \Z^k$ defined by a matrix $M$, one has
\begin{equation}\label{eq:FHZk}
R(\overline{\f})=\# \Coker (\Id -\overline{\f})=|\det(E-M)|,
\end{equation}
if $R(\overline{\f})<\infty$, and $|\det(E-M)|=0$ otherwise.

\section{Some classes of wreath products without $R_\infty$ property}

\begin{teo}\label{teo:cases}
Suppose, the prime-power decomposition of $G$ is $\oplus_i (\Z_{(p_i)^{r_i}})^{d_i}$. Then under each of the following 
conditions the corresponding wreath products $G\wr \Z^k$ admit an automorphism $\f$ with $R(\f)<\infty$, i.e. do not have the property $R_\infty$:
\begin{description}
\item[Case 1)] for all $p_i=2$ and $p_i=3$, we have  $d_i\ge 2$ (and is arbitrary for primes $>3$);
\item[Case 2)] there is no $p_i=2$ and $k$ is even;
\item[Case 3)] for all $p_i=2$, we have  $d_i\ge 2$  and $k=4s$ for some $s$.
\end{description}
\end{teo}

\begin{proof}
In each of these cases we will take an automorphism $\overline{\f}:\Z^k \to \Z^k$ with $R(\overline{\f})<\infty$ (in fact, of finite order)
and define $\f':\Sigma\to \Sigma$ with appropriate properties in accordance with Lemmas \ref{lem:R_needed} and \ref{lem:how_to_define}.

\textbf{Case 1).} In this case we can take $\overline{\f}=-\Id: \Z^k \to \Z^k$ and construct $\f$ similarly to \cite{gowon1}. More specifically, note that $R(\overline{\f})=2^k$
and
define $\f':\Sigma \to \Sigma$ in the following way.
The subgroups $G_x \oplus G_{-x}$ will be invariant subgroups of $\f'$ and we define
$$
\f': G_x \oplus G_{-x} \to G_x \oplus G_{-x}\mbox{ as } 
\begin{pmatrix}
0 & \Psi\\
\Psi & 0
\end{pmatrix},
$$
where $\Psi:G\to G$ is defined as a direct sum of blocks of the following types:
\begin{equation}\label{eq:F2F3}
F_2 = \begin{pmatrix}
			0 & 1\\
			1 & 1
			\end{pmatrix} : (\Z_q)^2 \to (\Z_q)^2, 
		\quad
		F_3 = \begin{pmatrix}
			0 & 0 & 1 \\
			0 & 1 & 1 \\
			1 & 1 & 1
		\end{pmatrix} : (\Z_q)^3 \to (\Z_q)^3, 
\end{equation}
where  $q$ are some $(p_i)^{r_i}$ and for each summand $\left(\Z_{(p_i)^{r_i}}\right)^{d_i}$ of $G$ ($d_i \ge 2$, $p_i=2$ or $p_i=3$) we have $s$ summands $F_2$, if $d_i=2s$, or $s-1$ summands $F_2$ and one summand $F_3$, if  $d_i=2s+1$.
For the remaining summands (i.e. for $p_i>3$) we do not need to group summands in the above way and we can consider $F_1:\Z_q\to \Z_q$,
$1 \mapsto m(q)$
 where $q=(p_i)^{r_i}$. This $m=m(q)$ should be taken in such a way that
\begin{equation}\label{eq:condit_on_m}
m^2 \mbox{ and } 1-m^2 \mbox{  are invertible in }\Z_q.
\end{equation}
This can be done for $p_i>3$: one can take $m=2$ (and impossible for $p_i=2$ or $3$). 

By Lemma \ref{lem:how_to_define}, we defined an automorphism $\f$ of $G\wr \Z^k$ in this way 
(one may assume (\ref{eq:restric_on_sub}) to have a unique $\f$).

We claim that $R(\t_{z_i} \circ \f')=R(\a(z_i) \circ \f')=1$, $i=1,\dots, 2^k$. Consequently, by Theorem \ref{teo:extensions}, $R(\f)=R(\overline{\f})=2^k<\infty$.
So we need to prove  that $\Id_\Sigma-\a(z_i) \circ \f'$ is an epimorphism, because, for Abelian groups, this is evidently the same
as $R(\a(z_i) \circ \f')=1$.   
This homomorphism has a decomposition of 
$\Sigma$ into invariant subgroups $G_{x}\oplus G_{-x +z_i}$,
because $\a(z_i): G_{-x} \to G_{-x +z_i}$, $\f': G_{-x +z_i} \to G_{x-z_i}$  and $\a(z_i): G_{x-z_i} \to G_{x}$. 
Note that the subgroups $G_{x}$ and $G_{-x +z_i}$ coincide if $z_i=2x$ (this corresponds to the case of $G_0$ for $\f'$). Thus it is sufficient to verify the epimorphity for each $G_{x}\oplus G_{-x +z_i}$ and for the exceptional case. Passing to summands of $G$, it is sufficient to verify the epimorphity of 
$$
\begin{pmatrix}
-E & F_2\\
F_2& -E
\end{pmatrix}, \quad
\begin{pmatrix}
-E & F_3\\
F_3& -E
\end{pmatrix}
\mbox{ and } 
\begin{pmatrix}
-E & F_1\\
F_1& -E
\end{pmatrix} = \begin{pmatrix}
-1 & m\\
m & -1
\end{pmatrix}.
$$   
The first two are isomorphisms with the explicit inverses 
$$
\left(
\begin{array}{cccc}
 -1 & 1 & 1 & 0 \\
 1 & 0 & 0 & 1 \\
 1 & 0 & -1 & 1 \\
 0 & 1 & 1 & 0 \\
\end{array}
\right), \qquad
\left(
\begin{array}{cccccc}
 -2 & 0 & 1 & 1 & 1 & -1 \\
 0 & -1 & 1 & 1 & 0 & 0 \\
 1 & 1 & -1 & -1 & 0 & 1 \\
 1 & 1 & -1 & -2 & 0 & 1 \\
 1 & 0 & 0 & 0 & -1 & 1 \\
 -1 & 0 & 1 & 1 & 1 & -1 \\
\end{array}
\right).
$$
For the  third one the invertibility follows from (\ref{eq:condit_on_m}).
For the exceptional case we formally do not need to verify the epimorphity, because it can add only a finite number to $R(\f')$,
but we wish to prove our (more strong) claim (this will be helpful for TBFT). So we have to prove, that
$$
F_2-E, \qquad F_3-E, \qquad m-1
$$
are epimorphisms. This can be done immediately: $\det(F_2-E)=1 \mod 2$,  $\det(F_3-E)=1 \mod 2$, and $1-m^2 = (1-m)(1+m)$.

\textbf{Case 2).} Now consider the case of even $k=2 t$ and $G$ without $2$-subgroup.
In this case the construction starts as in \cite{TroLamp}: 
we take $\overline{\f}:\Z^{2t} \to \Z^{2t}$ to be the direct sum of $t$ copies of
$$
\Z^2 \to \Z^2,\quad \begin{pmatrix}
u\\ v
\end{pmatrix} \mapsto M \begin{pmatrix}
u\\ v
\end{pmatrix}, \qquad M=\begin{pmatrix}
0 & 1\\
-1 & -1 
\end{pmatrix}. 
$$ 
Then $M$ generates a subgroup of $GL(2,\Z)$, which is isomorphic to $\Z_3$ (see, \cite[p. 179]{Newman1972book}).
All orbits of $M$ have length $3$ (except of the trivial one) and the corresponding Reidemeister number $=\det (E-M)=3$.
Similarly for $\overline{\f}$: the length of any orbit is $3$ (except of the zero) and $R(\overline{\f})=3^t$. 
Also
\begin{equation}\label{eq:sum_deg_M}
M^2 + M +E= \begin{pmatrix}
-1 & -1\\
1 & 0 
\end{pmatrix} + \begin{pmatrix}
0 & 1\\
-1 & -1 
\end{pmatrix}+ \begin{pmatrix}
1 & 0\\
0 & 1 
\end{pmatrix}= \begin{pmatrix}
0 & 0\\
0 & 0 
\end{pmatrix}.
\end{equation}

Now define $\f'$ as a direct sum of  actions for $\Z_q$, $q=(p_i)^{r_i}$, $p_i \ge 3$.

For $p_i \ge 3$ choose $m=m_i$ such that   
\begin{equation}\label{eq:condit_on_m_3}
m^3 \mbox{ and } 1-m^3 \mbox{  are invertible in }\Z_q.
\end{equation}
This can be done for $p_i\ge 3$: one can take $m=3$ for $p_i=7$ and $m=2$ in the remaining cases (and impossible for $p_i=2$).  
Define $\f'(a_0)=(m a)_0$ and 
$\f'(a_x)=(m a)_{\overline{\f}(g)}$, where $a \in \Z_q \subset G$. So, the corresponding subgroup $\oplus_{g \in \Z^k}(\Z_q)_g \subset \Sigma$ is $\f'$-invariant and decomposed into infinitely many invariant summands 
$(\Z_q)_g \oplus (\Z_q)_{\overline{\f}(g)}\oplus (\Z_q)_{\overline{\f}^2(g)}$
isomorphic to $(\Z_q)^3$ (over generic orbits of $\overline{\f}$) and one summand $(\Z_q)_0$ (over the trivial orbit).
Then the corresponding restrictions of $\f'$ and   $1-\f'$ can be written as multiplication by 
$$
\begin{pmatrix}
0  &  0 & m \\
m & 0 & 0\\
0 & m & 0
\end{pmatrix}, \quad
\begin{pmatrix}
1  &  0 & -m \\
-m & 1 & 0\\
0 & -m & 1
\end{pmatrix}, \quad \mbox{and }
m, \quad 1-m,
$$
respectively. The three-dimensional mappings are isomorphisms by (\ref{eq:condit_on_m_3}).  Since an element $\ell$ is not invertible in $\Z_{(p_i)^{r_i}}$ if and only if $\ell = u\cdot p_i$, the invertibility of one-dimensional mappings follows from 
(\ref{eq:condit_on_m_3}) and the factorization $1-m^3=(1-m)(1+m+m^2)$. (This construction gives a more explicit presentation of a part of proof 
of \cite[Theorem 4.1]{TroLamp})

For $\t_z \circ \f'$ we have
$$
(\t_z \circ \f') (g_x) = (m g)_{\overline{\f}(x)+z}, \quad
(\t_z \circ \f') (g_{\overline{\f}(x)+z})=(mg)_{\overline{\f}^2(x)+\overline{\f}z+z},
$$
$$
(\t_z \circ \f')g_{\overline{\f}^2(x)+\overline{\f}z+z}=(mg)_{\overline{\f}^3(x)+\overline{\f}^2z+\overline{\f}z+z}=(mg)_x,
$$
because  $\overline{\f}^3(x)=x$  and $\overline{\f}^2z+\overline{\f}z+z=0$  
by (\ref{eq:sum_deg_M}). So $\t_z \circ \f'$ has the same matrices as $\f'$, but on new invariant summands
$(\Z_q)_x \oplus (\Z_q)_{\overline{\f}(x)+z}\oplus (\Z_q)_{\overline{\f}^2(x)+\overline{\f}z+z}$. 
Similarly for the exceptional orbit.
This completes the proof of this case.

\textbf{Case 3):} when  $d_i >1$ for $p_i=2$ and $k = 4s $ .
Using the cyclotomic polynomial we can define (similarly to the above $M$) an element of order 5 in $GL(4,\Z)$ 
$$
M_4=
\begin{pmatrix}
0 & 0 & 0 & -1 \\
1 & 0  &  0 & -1 \\
0 & 1 &   0 & -1\\
0& 0&  1 & -1
\end{pmatrix}
$$
 (see e.g.  \cite{KuzmPavl2002} for an elementary introduction). 
For any  $k= 4s$, let $M\in GL(k,\Z)$ be the direct sum of  $s$ copies of $M_4$.
Let $\overline{\f}:\Z^k \to \Z^k$ be defined by $M$.  One can calculate
$$
\det(M_4-E)=5, \qquad \det(M-E)=5^s.
$$
Hence, by (\ref{eq:FHZk}),  $R(\overline{\f})=5^s <\infty$.
The length of any non-trivial orbit is $5$,  hence an \emph{odd} number. 

Similarly to $M$, one can verify that
\begin{equation}\label{eq:4M40}
(M_4)^4 +(M_4)^3+ (M_4)^2+ M_4 + E=0.
\end{equation}
This can be also deduced from the fact that the characteristic polynomial of  the ``companion matrix'' of a polynomial  $p$ is just $p$.

For $p$-power components $\Z_{p^i}$ with $p>2$, we define
$\f'$ (as above)  by $a_0 \mapsto (p-1) a_0$.  Then, for an orbit $u, \overline{\f} u, \dots, \overline{\f}^\gamma u$, we need to verify (for finiteness of $R(\f')$) that $(p-1)^\gamma$ as a homomorphism    $\Z_{p^i} \to \Z_{p^i}$ has no non-trivial fixed elements, i.e. $ (p-1)^\gamma \not\equiv 1 \mod p$. This is fulfilled because, for an odd $\gamma$, $(p-1)^\gamma -1 \equiv -2  \not\equiv 0 \mod p$.

For $2$-power components $\Z_{2^i}\oplus \Z_{2^i}$, we define
$\f'$  by $a_0 \mapsto F_2 a_0$ (as in (\ref{eq:F2F3})).  Then, for an orbit $u, \overline{\f} u, \dots, \overline{\f}^\gamma u$, we need to verify that $(F_2)^\gamma$ as a homomorphism   $\Z_{2^i}\oplus \Z_{2^i}\to \Z_{2^i}\oplus \Z_{2^i}$ has no non-trivial fixed elements.
Here we need to use not only the fact that $\gamma$ is odd, but its more specific form: $\g=5$.
In particular it can not be divided by $3=$order of $F_2$ $\mod 2$. Hence $(F_2)^\gamma=(F_2)^5=(F_2)^2=
\begin{pmatrix}
1& 1\\
1 & 0
\end{pmatrix}
\mod 2$.
 It has no non-trivial fixed elements $\mod 2^i$ for any $i$.

For $2$-power components $\Z_{2^i}\oplus \Z_{2^i}\oplus \Z_{2^i}$, we define
$\f'$  by $a_0 \mapsto F_3 a_0$ (as in (\ref{eq:F2F3})).  Then, for an orbit of $\overline{\f}$ of length $\gamma$, we need to verify that $(F_3)^\gamma$ as a homomorphism   $\Z_{2^i}\oplus \Z_{2^i}\oplus \Z_{2^i}\to \Z_{2^i}\oplus \Z_{2^i}\oplus \Z_{2^i}$ has no non-trivial fixed elements. One can verify, for $i=1$, i.e. for $2^i=2$, that the order of $F_3$ is relatively prime with $5$, namely it is equal to $7$. Moreover, $(F_3)^j$, $j=1,\dots,6$, has  no non-trivial fixed elements.  
The absence of non-trivial fixed elements is equivalent to $\det((F_3)^j-E) \not\equiv 0 \mod 2$. Then
$\det((F_3)^j-E) \not\equiv 0 \mod 2^i$. Hence, for $i>1$ these automorphisms still have no non-trivial fixed elements.
The elements $(F_3)^{7u}$, $u=1,2,\dots$, typically are not $E$ $\mod 2^i$, but in any case $7u \ne 5$, for any $u$. In fact, we are interested only in properties of $(F_3)^5$. 

Collecting together these homomorphisms defined on the summands, we obtain as in the first two cases, $\f'$ with the desired properties. It remains only to verify the epimorphity of $\t_z\circ \f'$. This can be done quite similarly to the end of Case 2) with the help of
(\ref{eq:4M40}).
\end{proof}

\section{Twisted Burnside-Frobenius Theorem}\label{sec:tbft}
\begin{teo}\label{teo:main_for_TBFT}
Suppose that $\f$ is an automorphism of the restricted wreath product $G\wr \Z^k=\oplus_{m\in \Z^k} G_m \rtimes_\a \Z^k$, where $G$ is a finite abelian group.
Suppose that $\f$ is of finite order. Then $R(\f')$ is $1$ or $\infty$.
\end{teo}

\begin{cor}\label{cor:main_TBFT}
	In particular, $\f$ has the TBFT$_f$ property. 
\end{cor}

\begin{proof}[Proof of Corollary]
By Lemma \ref{lem:R_needed}, $R(\f)<\infty$ implies
$R(\f')<\infty$. Hence, by Theorem \ref{teo:main_for_TBFT}, $R(\f')=1$.
Considering $\t_z \circ \f$ instead of $\f$ from the very beginning, we see that  $R(\t_z \circ \f')=1$, for any $z\in \Z^k$.
Thus, by Theorem \ref{teo:extensions}, Reidemeister classes $\{g\}_\f$ of $\f$ are pull-backs of Reidemeister classes
$\{z\}_{\overline{\f}} $ of $\overline{\f}$ under the natural projection $\pi:G\wr \Z^k \to \Z^k$, i.e. $\{g\}_\f=\pi^{-1}(\{\pi(g)\}_{\overline{\f}})$.  So, if classes of $\overline{\f}$  are separated by an epimorphism $f: \Z^k \to F$ onto a finite group $F$, then classes of $\f$ are separated by $f\circ\pi$. It remains to use the equivalence between TBFT$_f$ and
separability of Reidemeister classes in the case of finite Reidemeister number (see \cite{polyc} and
\cite{FelLuchTro}). 
\end{proof}

\begin{rmk}
	\rm
	In particular, this covers all automorphisms, which were considered in \cite{TroLamp}. Indeed, it was proved there, that all orbits are finite and
	their length is equal to the length of orbits of $\overline{\f}$. But the structure of $\Z^k$ implies that $\overline{\f}$ has finite order (consider generators). Hence, $\f'$ and $\f$ are of finite order.
\end{rmk}

\begin{proof}[Proof of Theorem]
	Suppose, $R(\f')>1$. Then there exists an element $\sigma\in \Sigma$ such that
	$\sigma\not \in \Im(\Id-\f')$. Moreover, $\sigma\not \in \Im(\Id-\f'_\sigma)$,
	where $\f'_\sigma$ is the restriction of $\f'$ onto the $\f'$-invariant subgroup $\Sigma_\sigma$ generated by $\sigma$. In particular, $R(\f'_\sigma)>1$. By the supposition $\Sigma_\sigma$ is a finite group with generators $\sigma,\f'(\sigma),\dots,(\f')^s(\sigma)$ for some $s$. Hence, $\f'_\sigma$ has
	a nontrivial fixed element $\sigma_0$, $\f'_\sigma(\sigma_0)=\sigma_0$ and
	$\sigma_0\ne 0$. For an element $m\in\Z^k$ consider the orbit
	$$
	\a(m)\sigma_0,\quad \f'(\a(m)\sigma_0)=\a(\overline{\f}(m))\sigma_0,\qquad
	(\f')^t(\a(m)\sigma_0)=\a(\overline{\f}^t(m))\sigma_0,\qquad
	\overline{\f}^{t+1}(m)=m. 
	$$ 
	Then $(\f')^{t+1}(\a(m)\sigma_0)=\a(m)\sigma_0$. Passing from $m$ to $nm$,
	$n\in \Z$, $m\in \Z^k$, if necessary, we can assume that the supports in $\Z^k$ of
	$\sigma_0$, $\a(\overline{\f}^j(nm))\sigma_0$, $j=0,\dots,t$, do not intersect. 
	Then $\sum_{j=1}^t \a(\overline{\f}^j(nm))\sigma_0$ is a fixed element of
	$\f'$, which is distinct from $0$ and $\sigma_0$. 
Increasing $n$ ``in sufficiently large steps'' we obtain infinitely many distinct fixed elements in the same way.	
	Then by Lemma \ref{lem:Jab_fin_ord}, $R(\f')=\infty$.
\end{proof}	

\def\cprime{$'$} \def\cprime{$'$} \def\cprime{$'$} \def\cprime{$'$}
  \def\cprime{$'$} \def\cprime{$'$} \def\cprime{$'$} \def\dbar{\leavevmode\hbox
  to 0pt{\hskip.2ex \accent"16\hss}d}
  \def\polhk#1{\setbox0=\hbox{#1}{\ooalign{\hidewidth
  \lower1.5ex\hbox{`}\hidewidth\crcr\unhbox0}}} \def\cprime{$'$}
  \def\cprime{$'$} \def\cprime{$'$}

\end{document}